\documentclass[11pt]{article}
\usepackage{amsmath, amssymb, amsthm}
\usepackage{hyperref}

\usepackage[a4paper,margin=1in]{geometry}

\newtheorem{theorem}{Theorem}
\newtheorem{lemma}{Lemma}
\newtheorem{corollary}{Corollary}

\newtheorem{definition}{Definition}

\newtheorem{note}{Note}

\title{The Number of Parts in the (Distinct) Partitions with Parts from a Set}
\author{A.~David Christopher \\ 
Department of Mathematics, The American College, Madurai, Tamil Nadu, India \\
\texttt{davidchristopher@americancollege.edu.in}
}
\date{} 

\begin{document}
\maketitle

\begin{abstract}
The number of parts in the partitions (resp.\ distinct partitions) of $n$ with parts from a set were considered. Its generating functions were obtained. Consequently, we derive several recurrence identities for the following functions: the number of prime divisors of $n$, $p$-adic valuation of $n$, the number of Carlitz-binary compositions of $n$ and the Hamming weight function. Finally, we obtain an asymptotic estimate for the number of parts in the partitions of $n$ with parts from a finite set of relatively prime integers.
\end{abstract}

\section{Motivation and Basic Definitions}

One can find a list of papers \cite{charlotte,erdos,hsien,lee,loxton,william} in the literature concerning the distribution of a given type of parts in a given class of partitions. 
In this article, we define the \emph{number of parts function} in two major classes of partitions, namely, in the partitions with parts from a given set and in the distinct partitions with parts from a given set. We derive their generating functions. All the results of this paper are obtained as the consequence of these generating functions. 
In Section~2, we present several recurrence identities for the following list of functions: the number of Carlitz-binary compositions of $n$, the number of prime divisors of $n$, $p$-adic valuation of $n$ and the Hamming weight function. 

In Section 3, we obtain an asymptotic estimate for the number of parts in the partitions of $n$ with parts from a finite set of relatively prime positive integers. 

The main classes of partitions and compositions of this paper are given in the following definition. 
\begin{definition}
\normalfont 
Let $n$ be a positive integer. By a partition (resp. composition) of $n$, we mean an unordered (resp. ordered) sequence of positive integers whose sum equals $n$. Each summand in the sum is called a part. 
\begin{enumerate}
\item 
A partition of $n$ is said to be a \emph{distinct} partition of $n$ if each part of it is different from the other parts. 
\item 
A partition (resp. composition) of $n$ is said to be a \emph{binary} partition (resp. composition) of $n$ if each part of it is of the form $2^k$ with integer $k\geq 0$.
\item 
A composition of $n$ is said to be a \emph{Carlitz} composition of $n$ if each part of it is different from its adjacent parts. 
\item 
A composition of $n$ is said to be a \emph{Carlitz-binary} composition of $n$ if it is both Carlitz and binary. 
\end{enumerate} 
\end{definition}
We introduce the \emph{number of parts function} which will be used for deriving several interesting results.  
\begin{definition}
\normalfont 
Let $n$ be a positive integer and let $A$ be a set of positive integers.
\begin{enumerate}
\item 
The function $N^p_A(n)$ is defined to be the number of parts in the partitions of $n$ with parts from the set $A$.
\item The function $N^q_A(n)$ is defined to be the number of parts in the distinct partitions of $n$ with parts from the set $A$.
\end{enumerate}
\end{definition}
The list of partition functions and divisor-sum functions involved in the recurrence identities of this paper are given in the following definition.
\begin{definition}
\normalfont 
Let $n$ be a positive integer and let $A$ be a set of positive integers. 
\begin{enumerate}
\item The function $p_A(n)$ (resp. $q_A(n)$) is defined to be the number of partitions (resp. distinct partitions) of $n$ with parts from the set $A$.
\item The function $p_A^e(n)$ (resp. $p_A^o(n)$) is defined to be the number of partitions of $n$ with even (resp. odd) number of parts from the set $A$.
\item  The function $q_A^e(n)$ (resp. $q_A^o(n)$) is defined to be the number of distinct partitions of $n$ with even (resp. odd) number of parts from the set $A$.
\item The function $\tau_A(n)$ is defined by 
\[\tau_A(n)=\sum_{\substack{a\mid n\\ a\in A}}1.
\]
\item The function $\tau_A^s(n)$ is defined by 
\[\tau_A^s(n)=\sum_{\substack{a\mid n\\ a\in A}}(-1)^{\frac{n}{a}-1}.
\]
\item The function $\sigma_A(n)$ is defined by 
\[\sigma_A(n)=\sum_{\substack{a\mid n\\ a\in A}}a.
\]
\item The function $\sigma_A^s(n)$ is defined by 
\[\sigma_A^s(n)=\sum_{\substack{a\mid n\\ a\in A}}(-1)^{\frac{n}{a}-1}a.
\]
\item The Hamming weight function, denoted $h(n)$, is defined to be the number of ones in the binary representation of $n$. 
\item Let $p$ be a prime number and let $q$ be a rational number. The $p$-adic valuation of $q$, denoted $\vartheta_p(q)$, is defined by $\vartheta_p(q)=r$, where $q=p^r\frac{a}{b}$ with $\gcd(a,b)=1$.
\end{enumerate}
\end{definition}
Throughout this article, we assume $x$ to be a real variable with $|x|<1$. 
It is well-known and easy to prove that
\begin{equation}\label{gf1}
\sum_{n=0}^{\infty}p_A(n)x^n=\prod_{a\in A}(1-x^a)^{-1}
\end{equation}
with $p_A(0)=1$,
\begin{equation}\label{gf2}
\sum_{n=0}^{\infty}q_A(n)x^n=\prod_{a\in A}(1+x^a)
\end{equation}
with $q_A(0)=1$,
\begin{equation}\label{gf3}
\sum_{n=0}^{\infty}\left(p_A^e(n)-p_A^o(n)\right)x^n=\prod_{a\in A}(1+x^a)^{-1}
\end{equation}
with $p_A^e(0)=1$ and $p_A^o(n)=0$
\newline 
and
\begin{equation}\label{gf4}
\sum_{n=0}^{\infty}\left(q_A^e(n)-q_A^o(n)\right)x^n=\prod_{a\in A}(1-x^a)
\end{equation}
with $q_A^e(0)=1$ and $q_A^o(0)=0$.

From the Lambert series expansion one can readily get the following two generating functions:
\begin{equation}\label{gf5}
\sum_{n=1}^{\infty}\tau_A(n)x^n=\sum_{a\in A}\frac{x^a}{1-x^a},
\end{equation}
\begin{equation}\label{gf6}
\sum_{n=1}^{\infty}\sigma_A(n)x^n=\sum_{a\in A}\frac{ax^a}{1-x^a}.
\end{equation}
Since $|x|<1$, we have
\[\frac{x^a}{1+x^a}=x^a-x^{2a}+x^{3a}-\cdots .
\]
This gives 
\[\sum_{a\in A}\frac{x^a}{1+x^a}=\sum_{a\in A}(x^a-x^{2a}+x^{3a}-\cdots).
\]
In the right side sum, every exponent of $x$ is of the form $n=ka$ for some $a\in A$. Thus, the coefficient of $x^n$ in this sum is 
\[
\sum_{k}(-1)^{k-1}=\sum_{\substack{a\in A\\ a\mid n}}(-1)^{\frac{n}{a}-1}=\tau_A^s(n).
\]
Consequently, we have
\begin{equation}\label{gf7}
\sum_{n=1}^{\infty}\tau_A^s(n)=\sum_{a\in A}\frac{x^a}{1+x^a}.
\end{equation}
Similar argument gives
\begin{equation}\label{gf8}
\sum_{n=1}^{\infty}\sigma_A^s(n)=\sum_{a\in A}\frac{ax^a}{1+x^a}.
\end{equation}

\section{ Generating Functions for $N_A^p(n)$ and $N_A^q(n)$, and Six ways to Convolution and Recurrence Identities}
\subsection{ Derivation with a Direct Application}
In this subsection, we derive generating functions for $N^p_A(n)$ and $N^q_A(n)$ by defining appropriate bijective maps. 
\begin{theorem}\label{iden-1}
Let $n$ be a positive integer and let $A$ be a set of positive integers. We have
\begin{description}
\item{(a)}
\begin{equation}
\sum_{n=1}^{\infty}N^p_A(n)x^n=\left(\prod_{a\in A}\frac{1}{1-x^a}\right)\left(\sum_{b\in A}\frac{x^b}{1-x^b}\right),
\end{equation}
\item{(b)}
\begin{equation}
N^p_A(n)=\sum_{k=0}^{n-1}p_A(k)\tau_A(n-k),
\end{equation}
\item{(c)}
\begin{equation}
\sum_{n=1}^{\infty}N^q_A(n)x^n=\left(\prod_{a\in A}(1+x^a)\right)\left(\sum_{b\in A}\frac{x^b}{1+x^b}\right),
\end{equation}
\item{(d)}
\begin{equation}
N^q_A(n)=\sum_{k=0}^{n-1}q_A(k)\tau_A^s(n-k).
\end{equation}
\end{description}
\end{theorem}
\begin{proof}
Let $b\in A$. Let $P_A^n$ be the set of partitions of $n$ with parts from the set $A$. Let $N_b^p(n)$ be the number of times $b$ occurs in $P_A^n$. Consider the mapping
\[(b_1,b_2,\cdots ,b_s)\to (b_1,b_2,\cdots ,b_s,b,b,\cdots ,(\text{k times})\ b)
\]
with $b_i\in A\setminus\{b\}$ and $b_1+b_2+\cdots +b_s+kb=n$.

We observe that this mapping establishes a one-to-one correspondence between $P_{A\setminus\{b\}}^{n-kb}$ and the set of partitions of $n$ with  part $b$ occurring exactly $k$ times and parts from the set $A$. This gives \[N_b^p(n)=p_{A\setminus\{b\}}(n-b)+2p_{A\setminus\{b\}}(n-2b)+3p_{A\setminus\{b\}}(n-3b)+\cdots.\]
Now, we have 
\begin{align*}
\sum_{n=1}^{\infty}N_b^p(n)x^n&=\sum_{n=1}^{\infty}\left(\sum_{k\geq 1}kp_{A\setminus\{b\}}(n-kb)\right)x^n\\ &=(x^b+2x^{2b}+\cdots)\left(\sum_{n=0}^{\infty}p_{A\setminus\{b\}}(n)x^n\right)\\
&=\frac{x^b}{(1-x^b)^2}\prod_{a\in A\setminus\{b\}}\frac{1}{1-x^a}\\
&=\frac{x^b}{1-x^b}\prod_{a\in A}\frac{1}{1-x^a}.
\end{align*}
Since $N_A^p(n)=\sum_{b\in A}N_b^p(n)$, we have
\begin{align*}
\sum_{n=1}^{\infty}N^p_A(n)x^n&=\sum_{n=1}^{\infty}\left(\sum_{b\in A}N_b^p(n)\right)x^n\\ &=\sum_{b\in A}\sum_{n=1}^{\infty} N_b^p(n)x^n\\ &=\sum_{b\in A}\left(\frac{x^b}{1-x^b}\prod_{a\in A}\frac{1}{1-x^a}\right)\\ &=\left(\prod_{a\in A}\frac{1}{1-x^a}\right)\left(\sum_{b\in A}\frac{x^b}{1-x^b}\right)
\end{align*}
as expected in (a).

Substituting (\ref{gf1}) and (\ref{gf5}) in (a), and equating the coefficients of $x^n$ on both sides gives (b).

Let $b\in A$. Let $Q_A^n$ be the set of all distinct partitions of $n$ with parts from the set $A$. Let $N_b^q(n)$ be the number of times $b$ occurs in $Q_A^n$. Consider the mapping
\[(b_1,b_2,\cdots,b_s)\to (b_1,b_2,\cdots ,b_s,b)
\]
with $b_i\in A\setminus \{b\}$ and $b_1+b_2+\cdots +b_s+b=n$. This mapping establishes a one-to-one correspondence between $Q_{A\setminus\{b\}}^{n-b}$ and the set of distinct partitions of $n$ in $Q_{A}^n$ having $b$ as a part. This gives $N_b^q(n)=|Q_{A\setminus\{b\}}^{n-b}|=q_{A\setminus \{b\}}(n-b)$. Now, we have
\begin{align*}
\sum_{n=1}^{\infty}N_b^q(n)x^n&=\sum_{n=1}^{\infty}q_{A\setminus\{b\}}(n-b)x^n
\\&=x^b\sum_{n=0}^{\infty}q_{A\setminus\{b\}}(n)x^n\\
\
&=x^b\prod_{a\in A\setminus\{b\}}(1+x^a) \\
&=\frac{x^b}{1+x^b}\prod_{a\in A}(1+x^a).
\end{align*}

Since $N_A^q(n)=\sum_{b\in A}N_b^q(n)$, we have
\begin{align*}
\sum_{n=1}^{\infty}N^q_A(n)x^n&=\sum_{n=1}^{\infty}\left(\sum_{b\in A}N_b^q(n)\right)x^n\\ &=\sum_{b\in A}\sum_{n=1}^{\infty}N_b^q(n)x^n\\ 
&=\sum_{b\in A}\left(\frac{x^b}{1+x^b}\prod_{a\in A}(1+x^a)\right)\\
&=\left(\prod_{a\in A}(1+x^a)\right)\left(\sum_{b\in A}\frac{x^b}{1+x^b}\right),
\end{align*}
as expected in (c).

Substituting (\ref{gf2}) and (\ref{gf7}) in (c), and equating the coefficients of $x^n$ on both sides gives (d). 
\end{proof}
As special cases of (b) of Theorem \ref{iden-1} we have the following three results. First, we have a convolution-sum expression for the number of parts in the (distinct) partitions of $n$. 
\begin{corollary}Let $N^p(n)$ (resp. $N^q(n)$) be the number of parts in the partitions (resp. distinct partitions) of $n$ and let $p(m)$ (resp. $q(m)$)  be the number of partitions (resp. distinct partitions) of $m$.
We have
\begin{description}
\item{(a)}
\begin{equation}
N^p(n)=\sum_{k=0}^{n-1}p(k)\tau(n-k),
\end{equation}
where $\tau(m)=\tau_{\mathbb N}(m)$,
\item{(b)} 
\begin{equation}
N^q(n)=\sum_{k=0}^{n-1}q(k)\tau^s(n-k),
\end{equation}
where $\tau^s(m)=\tau^s_{\mathbb N}(m)$.
\end{description}
\end{corollary}
As the second consequence of (b) of Theorem \ref{iden-1}, we have a recurrence identity for the number of prime divisors of $n$.
\begin{corollary}
Let $A=\{2,3,5,7,11,\cdots\}$ be the set of all prime numbers. Then we have
\begin{equation}
N^p_A(n)=\sum_{k=0}^{n-1}p_A(k)\Omega(n-k),
\end{equation} 
where $\Omega(m)$ denotes the number of prime divisors of $m$.
\end{corollary}
As the third consequence of (b) of Theorem \ref{iden-1}, we have a recurrence identity for the $p$-adic valuation of $n$ for prime $p$. 
\begin{corollary}
Let $p$ be a prime number and let $A=\{1,p,p^2,p^3,\cdots \}$. We have
\begin{equation}
N^p_A(n)=\sum_{k=0}^{n-1}p_A(k)\left(\vartheta_p(n-k)+1\right).
\end{equation}
\end{corollary}
As a consequence of (d) of Theorem \ref{iden-1}, we have an elegant recurrence relation for the Hamming weight function.
\begin{corollary}\label{ham-rec}
For integer $n\geq 2$, we have
\begin{equation}
h(n)=h(n-1)+1-\vartheta_2(n).
\end{equation}
\end{corollary}
\begin{proof}
Set $A=\{1,2,2^2,2^3,\cdots \}$. Then by basis representation theorem, we have $q_A(n)=1$ for every $n\geq 1$. By the definition: $q_A(0)=1$. Also we have $\tau_A^s(n)=1-\vartheta_2(n)$ and $N_A^q(n)=h(n)$. Substituting these values in (d) of Theorem \ref{iden-1} gives 
\begin{align*}
h(n)&=\sum_{k=1}^{n}(1-\vartheta_2(k))\\
&=h(n-1)+1-\vartheta_2(n)
\end{align*}
as expected. 
\end{proof}
\begin{note}
Corollary \ref{ham-rec} can be written as $h(n)=\vartheta_2\left(\frac{2^{n}}{n!}\right)$.
\end{note}
\subsection{ Inversion}
The following result is an inversion formula for (b) and (d) of Theorem \ref{iden-1}.
\begin{theorem}\label{iden-2}
Let $n$ be a positive integer and let $A$ be a set of positive integers. We have 
\begin{description}
\item{(a)}
\begin{equation}
\sum_{k=1}^{n}N^p_A(k)\left(q_A^e(n-k)-q_A^o(n-k)\right)=\tau_A(n),
\end{equation}
\item{(b)}
\begin{equation}
\sum_{k=1}^nN^q_A(k)\left(p_A^e(n-k)-p_A^o(n-k)\right)=\tau_A^s(n).
\end{equation}
\end{description}
\end{theorem}
\begin{proof}
From (a) of Theorem \ref{iden-1}, we have
\begin{equation}\label{man-iden-1}
\left(\sum_{n=1}^{\infty}N^p_A(n)x^n\right)\left(\prod_{a\in A}(1-x^a)\right)=\sum_{b\in A}\frac{x^b}{1-x^b}.
\end{equation}
Substituting (\ref{gf4}) and (\ref{gf5}) in (\ref{man-iden-1}), and equating the coefficients of $x^n$ on both sides gives (a). 

From (c) of Theorem \ref{iden-1}, we have
\begin{equation}\label{man-iden-2}
\left(\sum_{n=1}^{\infty}N^q_A(n)x^n\right)\left(\prod_{a\in A}\frac{1}{1+x^a}\right)=\sum_{b\in A}\frac{x^b}{1+x^b}.
\end{equation}
Substituting (\ref{gf3}) and (\ref{gf7}) in (\ref{man-iden-2}), and equating the coefficients of $x^n$ on both sides gives (b). 
\end{proof}

A recurrence identity of the form $g(n)=\sum_{k=0}^{n}\omega(k)f(n-k)$ with 
\[\omega(m)=\begin{cases}1 &\text{ if }m=0;\\ 
(-1)^k &\text{ if }m=\frac{3k^2\pm k}{2}
\end{cases}
\]
is called an Euler-type recurrence identity. 
As a consequence of Theorem \ref{iden-2} we have an Euler-type recurrence identity for the number of parts in the partitions of $n$. 
\begin{corollary}
We have
\begin{description}
\item{(a)}
\begin{equation}\label{NOP}
\sum_{k=1}^{n}N^p(k)\omega(n-k)=\tau(n),
\end{equation}
\item{(b)}
\begin{equation}
\sum_{k=1}^{n}(-1)^{n-k}N^q(k)o(n-k)=\tau^s(n),
\end{equation}
where $o(m)$ denotes the number of distinct partitions of $m$ with odd parts, $\tau^s(m)=\tau_{\mathbb N}^s(m)$ and $\tau(m)=\tau_{\mathbb N}(m)$. 
\end{description}
\end{corollary}
\begin{proof}
Set $A=\mathbb N$. Then by Euler's pentagonal number theorem \cite{euler-pnt} we have \[
q_A^e(n-k)-q_A^o(n-k)=\omega(n-k).\] Substituting this value in (a) of Theorem \ref{iden-2}, and equating the coefficients of $x^n$ on both sides gives (a).

From (\ref{gf3}) we have
\[\prod_{n=1}^{\infty}\frac{1}{1+x^n}=\sum_{m=0}^{\infty}(p_{\mathbb N}^e(m)-p_{\mathbb N}^o(m))x^m.
\]
In view of Euler's partition theorem \cite{euler}, we have
\begin{equation}\label{ept}
\prod_{n=1}^{\infty}\frac{1}{1+x^n}=\prod_{n=1}^{\infty}(1-x^{2n-1})=\sum_{m=0}^{\infty}(-1)^mo(m)x^m.
\end{equation}
Substituting these values in (b) of Theorem \ref{iden-2}, and equating the coefficients of $x^n$ on both sides gives (b). 
\end{proof}
As another application of Theorem \ref{iden-2} we express the number of prime divisors of $n$ as the convolution sum of $N^p_A(k)$ and $q_A^e(k)-q_A^o(k)$ with $A$ being the set of prime numbers.
\begin{corollary}
Let $A=\{2,3,5,7,11,\cdots\}$ be the set of all prime numbers. We have
\begin{equation}
\sum_{k=1}^nN^p_A(k)(q_A^e(n-k)-q_A^o(n-k))=\Omega(n).
\end{equation}
\end{corollary}
A result similar to the above one is achieved for the $p$-adic valuation of $n$.   
\begin{corollary}\label{cor-3-iden-2}
Let $p$ be a prime number and let $A=\{1,p,p^2,p^3,\cdots \}$. We have
\begin{equation}
\sum_{k=1}^nN^p_A(k)(q_A^e(n-k)-q_A^o(n-k))=\vartheta_p(n)+1.
\end{equation}
\end{corollary}
As a special case of Corollary \ref{cor-3-iden-2} we have a relation between Hamming weight function, the number of parts in the binary partitions of $n$ and the $2$-adic valuation of $n$.
\begin{corollary}
Let $N^p_{bin}(m)$ be the number of parts in the binary partitions of $m$. Define $s(0)=1$ and $s(m)=(-1)^{h(m)}$ for $m\geq 1$. Then we have 
\begin{equation}
\sum_{k=1}^{n}N^p_{bin}(k)s(n-k)=\vartheta_2(n)+1.
\end{equation}
\end{corollary}
\subsection{ Product with $\prod_{n=1}^{\infty}(1-x^n)$ and $\prod_{n=1}^{\infty}(1+x^n)^{-1}$}
Multiplying the terms $\prod_{n=1}^{\infty}(1-x^n)$ and $\prod_{n=1}^{\infty}(1+x^n)^{-1}$, respectively, with the generating functions (a) and (c) of Theorem \ref{iden-1} gives recurrence identities (separately)  for $N^p_A(m)$ and $N^q_A(m)$ when $A\neq\mathbb N$, of which one is Euler-type.
\begin{theorem}\label{iden-3}
We have
\begin{description}
\item{(a)}
\begin{equation}
\sum_{k=1}^{n}N^p_A(k)\omega(n-k)=\sum_{k=1}^{n}\tau_A(k)\left(q_{\mathbb N-A}^e(n-k)-q_{\mathbb N-A}^o(n-k)\right),
\end{equation}
\item{(b)}
\begin{equation}
\sum_{k=1}^n(-1)^{n-k}N^q_A(k)o(n-k)=\sum_{k=1}^{n}\tau_A^s(k)\left(p_{\mathbb N-A}^e(n-k)-p_{\mathbb N-A}^o(n-k)\right).
\end{equation}
\end{description} 
\end{theorem}
\begin{proof}
From (a) of Theorem \ref{iden-1}, we have
\[\sum_{n=1}^{\infty}N^p_A(n)x^n=\left(\prod_{a\in A}\frac{1}{1-x^a}\right)\left(\sum_{b\in A}\frac{x^b}{1-x^b}\right).
\]
Multiplying both sides with $\prod_{m=1}^{\infty}(1-x^m)$ gives 
\[\left(\sum_{n=1}^{\infty}N^p_A(n)x^n\right)\left(\prod_{m=1}^{\infty}(1-x^m)\right)=\left(\prod_{a\in\mathbb N\setminus A}(1-x^a)\right)\left(\sum_{b\in A}\frac{x^b}{1-x^b}\right).
\]
Now in view of Euler's pentagonal number theorem \cite{euler-pnt}, (\ref{gf4}) and (\ref{gf5}), we get (a).

From (c) of Theorem \ref{iden-1}, we have
\[\sum_{n=1}^{\infty}N^q_A(n)x^n=\left(\prod_{a\in A}(1+x^a)\right)\left(\sum_{b\in A}\frac{x^b}{1+x^b}\right).
\]
Multiplying both sides by $\prod_{m=1}^{\infty}\frac{1}{1+x^m}$ gives 
\begin{equation}\label{prod-lemma}
\left(\sum_{n=1}^{\infty}N^q_A(n)x^n\right)\left(\prod_{m=1}^{\infty}\frac{1}{1+x^m}\right)=\left(\prod_{a\in \mathbb N\setminus A}\frac{1}{1+x^a}\right)\left(\sum_{b\in A}\frac{x^b}{1+x^b}\right).
\end{equation}
Now substituting (\ref{ept}), (\ref{gf3}) and (\ref{gf7}) in the equation above, and equating the coefficients of like powers of $x$ on both sides gives (b).  
\end{proof}
\subsection{ Logarithmic Differentiation}
Logarithmic differentiation of the generating function of $p(n)$ gives the following relation:
\begin{equation}\label{sigma}
np(n)=\sum_{k=0}^{n-1}\sigma(k)p(n-k).
\end{equation}
We wield the same technique for $N^p_A(n)$ and $N^q_A(n)$ to get results of similar kind.
\begin{theorem}\label{iden-4}
We have
\begin{description}
\item{(a)}
\begin{equation}
nN^p_A(n)=\sum_{k=1}^{n-1}N^p_A(k)\sigma_A(n-k)+\sum_{t=1}^{n}t\tau_A(t)p_A(n-t),
\end{equation}
\item{(b)}
\begin{equation}
nN^q_A(n)=\sum_{k=1}^{n-1}N^q_A(k)\sigma_A^s(n-k)+\sum_{t=1}^{n}t\tau_A^s(t)q_A(n-t).
\end{equation}
\end{description}
\end{theorem}
\begin{proof}
Define the following:
\[N^p_A(x)=\left(\prod_{a\in A}\frac{1}{1-x^a}\right)\left(\sum_{b\in A}\frac{x^b}{1-x^b}\right);
\]
\[N^q_A(x)=\left(\prod_{a\in A}(1+x^a)\right)\left(\sum_{b\in A}\frac{x^b}{1+x^b}\right);
\]
\[p_A(x)=\prod_{a\in A}\frac{1}{(1-x^a)};
\]
\[q_A(x)=\prod_{a\in A}(1+x^a);
\]
\[\tau_A(x)=\sum_{b\in A}\frac{x^b}{1-x^b};
\]
\[\tau_A^s(x)=\sum_{b\in A}\frac{x^b}{1+x^b}.
\]
Now taking the logarithm of $N^p_A(x)$ and then differentiating gives 
\[
\frac{N^{p^\prime}_A(x)}{N^p_A(x)}=\sum_{a\in A}\frac{ax^{a-1}}{1-x^a}+\frac{\tau_A^\prime(x)}{\tau_A(x)}.
\]
Multiplying both sides by $x$ and rearranging gives 
\begin{align*}
xN^{p^\prime}_A(x)&=N^p_A(x)\left(\sum_{a\in A}\frac{ax^{a}}{1-x^a}+\frac{x\tau_A^\prime(x)}{\tau_A(x)}\right)\\
&=N^p_A(x)\left(\sum_{a\in A}\frac{ax^{a}}{1-x^a}+\frac{x\tau_A^\prime(x)}{N^p_A(x)}p_A(x)\right)\\
&=N^p_A(x)\left(\sum_{a\in A}\frac{ax^{a}}{1-x^a}\right)+x\tau_A ^\prime(x)p_A(x).
\end{align*}
Now equating the coefficients of $x^n$ on extreme terms of the above chain of equalities gives (a).

Taking the logarithm of $N^q_A(x)$ and differentiating gives
\[\frac{N^{q^\prime}_A(x)}{N^q_A(x)}=\sum_{a\in A}\frac{ax^{a-1}}{1+x^a}+\frac{\tau_A^{s^\prime}(x)}{\tau_A^s(x)}.
\]
Multiplying both sides by $x$ gives
\begin{align*}
xN^{q^\prime}_A(x)&=N^q_A(x)\left(\sum_{a\in A}\frac{ax^{a}}{1+x^a}+\frac{x\tau_A^{s^\prime}(x)}{\tau_A^s(x)}\right)\\
&=N^q_A(x)\left(\sum_{a\in A}\frac{ax^{a}}{1+x^a}+\frac{x\tau_A^{s^\prime}(x)}{N^q_A(x)}q_A(x)\right)\\ 
&=N^q_A(x)\left(\sum_{a\in A}\frac{ax^{a}}{1+x^a}\right)+x\tau_A^{s^\prime}(x)q_A(x).
\end{align*}
Now equating the coefficients of $x^n$ on extreme terms of the above chain of equalities gives (b).
\end{proof}
\begin{corollary}
We have
\begin{description}
\item{(a)}
\begin{equation}
nN^p(n)=\sum_{k=1}^{n-1}N^p(k)\sigma(n-k)+\sum_{t=1}^{n}t\tau(t)p(n-t),
\end{equation}
\item{(b)}
\begin{equation}
nN^q(n)=\sum_{k=1}^{n-1}N^q(k)\sigma^s(n-k)+\sum_{t=1}^{n}t\tau^s(t)q(n-t),
\end{equation}
\item{(c)}
\begin{equation}
nN_{bin}^p(n)=\sum_{k=1}^{n-1}N_{bin}^p(k)\left(2^{\vartheta_2(n-k)+1}-1\right)+\sum_{t=1}^nt(\vartheta_2(t)+1)b(n-t).
\end{equation}
\end{description}
\end{corollary}
\subsection{ Recurrence Identities over $A$}
In this section, we derive recurrence identities for $N_A^p(n)$ and $N_A^q(n)$ over the set $A$. 
\begin{theorem}\label{roverA}
Let $n$ be a positive integer and let $A$ be a set of positive integers. For $s\in A$, we have
\begin{equation}
N_A^p(n)-N_A^p(n-s)=p_A(n-s)+N_{A\setminus\{s\}}^p(n)
\end{equation}
with $p_A(0)=1$.
\end{theorem}
\begin{proof}
By Theorem \ref{iden-1},
\begin{equation}
\sum_{n=1}^{\infty}N_A^p(n)x^n=\left(\prod_{a\in A}\frac{1}{1-x^a}\right)\left(\sum_{b\in A}\frac{x^b}{1-x^b}\right).
\end{equation}
This gives 
\begin{align*}
\left(\sum_{n=1}^{\infty}N_A^p(n)x^n\right)(1-x^s)&=\left(\prod_{a\in A\setminus\{s\}}\frac{1}{1-x^a}\right)\left(\sum_{b\in A}\frac{x^b}{1-x^b}\right)\\
&=\left(\prod_{a\in A\setminus\{s\}}\frac{1}{1-x^a}\right)\left(\frac{x^s}{1-x^s}+\sum_{b\in A\setminus\{s\}}\frac{x^b}{1-x^b}\right)\\
&=x^s\prod_{a\in A}\frac{1}{1-x^a}+\left(\prod_{a\in A\setminus\{s\}}\frac{1}{1-x^a}\right)\left(\sum_{b\in A\setminus\{s\}}\frac{x^b}{1-x^b}\right).
\end{align*}
Now writing the products in the right side as infinite sums in accordance with (\ref{gf1}) and (a) of Theorem \ref{iden-1}, and equating the coefficients of $x^n$ on both sides gives the expected end. 
\end{proof}
\begin{theorem}
Let $n$ be a positive integer and let $A$ be a set of positive integers. Let $s\in A$. We have
\begin{equation}
N_A^q(n)-N_A^q(n-s)+N_A^q(n-2s)+\cdots =N_{{A\setminus\{s\}}}^q(n)+q_{{A\setminus\{s\}}}(n-s)-q_{{A\setminus\{s\}}}(n-2s)+\cdots 
\end{equation}
\end{theorem}
\begin{proof}
By Theorem \ref{iden-1},
\begin{align*}
\frac{1}{1+x^s}\sum_{n=1}^{\infty}N^q_A(n)x^n=&\left(\prod_{a\in A\setminus\{s\}}(1+x^a)\right)\left(\sum_{b\in A\setminus\{s\}}\frac{x^b}{1+x^b}\right)\\&+\frac{x^s}{1+x^s}\left(\prod_{a\in A\setminus\{s\}}(1+x^a)\right).
\end{align*}
Now writing the products in the right side as infinite sums in accordance with (\ref{gf2}) and (c) of Theorem \ref{iden-1}, and equating the coefficients of $x^n$ on both sides gives the expected end.
\end{proof}
\subsection{ Interplay with Carlitz Compositions}
In this subsection we find some connections between the number of Carlitz's compositions, 2-adic valuation and the number of parts in the distinct partitions  in terms of finite discrete convolutions.
\begin{theorem}\label{iden-5}
Let $cl_A(m)$ be the number of Carlitz compositions of $m$. We have
\begin{description}
\item{(a)}
\begin{equation}
cl_A(n)=\sum_{k=0}^{n-1}cl_A(k)\tau_A^s(n-k),
\end{equation}
\item{(b)}
\begin{equation}
\sum_{k=0}^{n}cl_A(k)q_A(n-k)-\sum_{t=0}^{n-1}cl_A(t)N^q_A(n-t)=q_A(n)
\end{equation}
with $cl_A(0)=1$.
\end{description}
\end{theorem}
\begin{proof}
Heubach and Mansour \cite{heubach} derived the following generating function for the number of Carlitz compositions of $n$ with parts from the set $A$:
\begin{equation}\label{carl}
\sum_{n=0}^{\infty}cl_A(n)x^n=\frac{1}{1-\sum_{a\in A}\frac{x^a}{1+x^a}},
\end{equation}
which can be written as
\[\left(\sum_{n=0}^{\infty}cl_A(n)x^n\right)\left(1-\sum_{a\in A}\frac{x^a}{1+x^a}\right)=1.
\]
The coefficient of $x^n$ on the left side term is 
\[cl_A(n)-\sum_{k=0}^{n-1}cl_A(k)\tau_A^s(n-k).
\]
Since the coefficient of $x^n$ on the right side is zero, we have (a).

Multiplying both sides of (\ref{carl}) by $\frac{1}{\prod_{a\in A}(1+x^a)}$ gives
\begin{center}
 $\Big(\sum_{n=0}^{\infty}cl_A(n)x^n\Big)$ $\left(\frac{1}{\prod_{a\in A}(1+x^a)}\right)=\frac{1}{\prod_{a\in A}(1+x^a)-\left(\prod_{a\in A}(1+x^a)\right)\left(\sum_{a\in A}\frac{x^a}{1+x^a}\right)}.$
\end{center}
This can be arranged as
\[\left(\sum_{n=0}^{\infty}cl_A(n)x^n\right)\left[\left(\prod_{a\in A}(1+x^a)\right)-\left(\prod_{a\in A}(1+x^a)\right)\left(\sum_{a\in A}\frac{x^a}{1+x^a}\right)\right]=\prod_{a\in A}(1+x^a).
\]
Then in view of (c) of Theorem \ref{iden-1} and (\ref{gf2}) we can write
\[\left(\sum_{n=0}^{\infty}cl_A(n)x^n\right)\left[\sum_{l=0}^{\infty}q_A(l)x^l-\sum_{m=1}^{\infty}N^q_A(m)x^m\right]=\sum_{k=0}^{\infty}q_A(k)x^k.
\]
Now equating the coefficients of like powers of $x$ on both sides gives (b).
\end{proof}
\begin{corollary}
let $cl_b(m)$ be the number of Carlitz-binary compositions of $m$. We have
\begin{description}
\item{(a)}
\begin{equation}
cl_b(n)=\sum_{k=0}^{n-1}cl_b(k)(1-\vartheta_2(n-k)),
\end{equation}
\item{(b)}
\begin{equation}
\sum_{k=0}^{n}cl_b(k)\left(1-h(n-k)\right)=1
\end{equation}
with $h(0)=0$.
\end{description}
\end{corollary}
\begin{proof}
Let $A=\{1,2,2^2,\cdots\}$. Then in view of basis representation theorem we have $q_A(m)=1$ and $N_A^q(m)=h(m)$ for every positive integer $m$. Moreover, we have $\tau^s_A(m)=1-\vartheta_2(m)$.  The result is immediate while we substitute these values in Theorem \ref{iden-5}.
\end{proof}
\section{ Asymptotic Estimate of $N_A^p(n)$ when $A$ is a Finite Set}
Throughout this section we assume $A=\{a_1,a_2,\cdots ,a_k\}$, a finite set of positive integers with $\gcd(A)=1$. 
Based on this assumption we find an asymptotic estimate for $N_A^p(n)$.
To that end, we take cue from an earlier paper of the author \cite{david1}, where the function $p_A(n)$ was considered. It was shown there that $p_A(a_1a_2\cdots a_kl+r)$ is a polynomial in $l$ of degree $k$ for each $r\in\{0,1,2,\cdots ,a_1a_2\cdots a_k-1\}$ with the identical leading coefficient $\frac{(a_1a_2\cdots a_k)^{k-2}}{(k-1)!}$. This fact was used in \cite{david1} to arrive at the estimate of Netto \cite{Netto}:
\[p_A(n)\sim\frac{n^{k-1}}{(a_1a_2\cdots a_k)(k-1)!}.
\]

A quasi-polynomial representation of $p_A(n)$  found in \cite{david1} is recalled in the following lemma.
\begin{lemma}\label{lemma2}
Let $l$ be a non-negative integer and let $A=\{a_1,a_2,\cdots ,a_k\}$ be a set of positive integers such that $\gcd(A)=1$. For each $r\in\{0,1,\cdots,a_1a_2\cdots a_k-1\}$, the term $p_A(a_1a_2\cdots a_kl+r)$ is a polynomial in $l$ of degree $k-1$ with the leading coefficient $\frac{(a_1a_2\cdots a_k)^{k-2}}{(k-1)!}$.
\end{lemma}
The main result of this section is given below.
\begin{theorem}
Let $A=\{a_1,a_2,\cdots ,a_k\}$ be a set of positive integers such that $\gcd(A)=1$. Then we have
\begin{equation}\label{mest}
N_A^p(n)\sim\frac{1}{k!}\frac{\frac{1}{a_1}+\frac{1}{a_2}+\cdots +\frac{1}{a_k}}{a_1a_2\cdots a_k}n^k.
\end{equation}
\end{theorem}
\begin{proof}
The crux of this proof is to show that $N_A^p(a_1a_2\cdots a_kl+r)$ is a polynomial in $l$ of degree $k$ with the leading coefficient \[\frac{\sum_{i=1}^{k}\left(a_i^{k-2}\prod_{\substack{j\neq i\\ 0\leq j\leq k}}a_j^{k-1}\right)}{k!}\] for each $r\in\{0,1,\cdots ,a_1a_2\cdots a_k-1\}$. Once this is established, then we will have
\[\lim_{l\to\infty}\frac{N^p_A(a_1a_2\cdots a_kl+r)}{(a_1a_2\cdots a_kl+r)^{k}}=\frac{1}{k!}\frac{\frac{1}{a_1}+\frac{1}{a_2}+\cdots +\frac{1}{a_k}}{a_1a_2\cdots a_k}
\]
for each $r\in\{0,1,\cdots ,a_1a_2\cdots a_k-1\}$, which is equivalent to the asymptotic estimate (\ref{mest}).

We will establish this main claim using induction over $|A|=k$. Since the term $a_i^{k-2}$ is involved in the leading coefficient, we take $k=3$ as the initial case for induction. Nevertheless the cases $k=1,2$ are needed to realise the case $k=3$. 

When $|A|=1$ the only way out is $a_1=1$. In this case we have $N_A^p(n)=n$, and the targeted estimate follows. To proceed further, we need the following observation: 
\[
N_{\{a\}}^p(n)=\begin{cases}\frac{n}{a} &\text{ if }a\mid n;\\
0&\text{ otherwise. }
\end{cases}
\]
Assume  $|A|=2$. That is, $A=\{a_1,a_2\}$ with $\gcd(a_1,a_2)=1$.
Fix $r\in\{1,2,\cdots ,a_1\}$. Applying Theorem \ref{roverA} $a_1$ times, we get
\begin{align*}
N_A^p(a_1a_2l+r)-N_A^p(a_1a_2(l-1)+r)&=\sum_{m=1}^{a_1}p_A(a_1a_2l+r-ma_2)\\ &\ \ \ +\sum_{t=0}^{a_1-1}N_{\{a_1\}}^p(a_1a_2l+r-ta_2).
\end{align*}
Since $\gcd(a_1,a_2)=1$, the congruence equation $a_2t\equiv r\pmod {a_1}$ has a unique solution modulo $a_1$, say $t^*$. This gives
\begin{align*}
\sum_{t=0}^{a_1-1}N_{\{a_1\}}^p(a_1a_2l+r-ta_2)&=\frac{a_1a_2l+r-t^*a_2}{a_1}\\ 
&=a_2l+\frac{r-t^*a_2}{a_1}\\
&=a_2l+k_1
\end{align*}
for each $r\in\{0,1,\cdots ,a_1-1\}$, where $k_1$ is an integer constant.

Next, we analyze the sum $\sum_{m=1}^{a_1}p_A(a_1a_2l+r-ma_2)$. Let $m\in\{1,2,\cdots ,a_1\}$. \newline 
Case i. Assume $r-ma_2<0$. In this case, consider the term $p_A(a_1a_2l+r-ma_2)$, which can be written as $p_A(a_1a_2(l-1)+a_1a_2+r-ma_2)$. Here we note that $0\leq a_1a_2+r-ma_2\leq a_1a_2-1$. Now, in view of Lemma \ref{lemma2}, we have
\begin{align*}
p_A(a_1a_2(l-1)+a_1a_2+r-ma_2)=(l-1)+c_m,
\end{align*}
where $c_m$ is an integer constant.\newline 
Case ii. Assume $r-ma_2\geq 0$. Then by Lemma \ref{lemma2}, we have $p_A(a_1a_2l+r-ma_2)=l+d_m$,
where $d_m$ is an integer constant.

In both the cases, $p_A(a_1a_2l+r-ma_2)=l+b_m$ for some constant $b_m$. This gives
\begin{align*}
\sum_{m=1}^{a_1}p_A(a_1a_2l+r-ma_2)&=\sum_{m=1}^{a_1}(l+b_m)\\
&=a_1l+k_2,
\end{align*}
where $k_2=\sum_{m=1}^{a_1}b_m$ is an integer constant. \newline 
Consequently,
\[N_A^p(a_1a_2l+r)-N_A^p(a_1a_2(l-1)+r)=(a_1+a_2)l+c,
\]
where $c=k_1+k_2$ is an integer constant.

This gives
\[N_A^p(a_1a_2l+r)=(a_1+a_2)\left(\frac{l^2+l}{2}\right)+cl+N_A^p(r).
\]
Subsequently, we have
\[\lim_{l\to\infty}\frac{N_A^p(a_1a_2l+r)}{(a_1a_2l+r)^2}
=\frac{\frac{1}{a_1}+\frac{1}{a_2}}{2!(a_1a_2)}.
\]
Since this is true for each $r\in\{0,1,\cdots ,a_1a_2-1\}$, the targeted estimate follows for the case $k=2$.

Now we have the initial case of the induction, that is, $k=3$. Assume $A=\{a_1,a_2,a_3\}$ with $\gcd(a_1,a_2,a_3)=1$. Fix $r\in\{0,1,\cdots ,a_1a_2a_3-1\}$. 

Let $d=\gcd(a_1,a_2)$. Applying Theorem \ref{roverA} $a_1a_2$ times, we get
\begin{align*} 
&N_A^p(a_1a_2a_3l+r)-N_A^p(a_1a_2a_3(l-1)+r)\\ 
&=\sum_{t=1}^{a_1a_2}p_A(a_1a_2a_3l+r-ta_3)+\sum_{m=0}^{a_1a_2-1}N_{A\setminus \{a_3\}}^p(a_1a_2a_3l+r-ma_3)\ \ \ \ \ \ \ \ \ \ \ \ \ \ \ \ \ \ \ \ \ \ \\ 
&=\sum_{t=1}^{a_1a_2}p_A(a_1a_2a_3l+r-ta_3)+\sum_{\substack{0\leq m\leq a_1a_2-1\\ d\mid r-ma_3}}N_{ \{a_1,a_2\}}^p\left(\frac{a_1a_2a_3l}{d}+\frac{r-ma_3}{d}\right)
\end{align*}
\begin{align}\label{eqn1}
=\sum_{t=1}^{a_1a_2}p_A(a_1a_2a_3l+r-ta_3)+\sum_{\substack{0\leq m\leq a_1a_2-1\\ d\mid r-ma_3}}N^p_{\{\frac{a_1}{d},\frac{a_2}{d}\}}\left(\frac{a_1}{d}\frac{a_2}{d}\frac{a_3d^2}{d}l+\frac{r-ma_3}{d}\right).\ \ 
\end{align}
If $r-ta_3<0$, then by Lemma \ref{lemma2}, $p_A(a_1a_2a_3(l-1)+a_1a_2a_3+r-ta_3)$ is a polynomial in $l-1$ of degree $2$ with the leading coefficient $\frac{a_1a_2a_3}{2!}$. If $r-ta_3\geq 0$, then again by Lemma \ref{lemma2}, $p_A(a_1a_2a_3l+r-ta_3)$ is a polynomial in $l$ of degree $2$ with the leading coefficient $\frac{a_1a_2a_3}{2!}$. We observe that, in both the cases, $p_A(a_1a_2a_3l+r-ta_3)$ is a polynomial in $l$ of degree $2$ with the leading coefficient $\frac{a_1a_2a_3}{2!}$. This in turn gives $\sum_{t=1}^{a_1a_2}p_A(a_1a_2a_3l+r-ta_3)$ is a polynomial in $l$ of degree 2 with the leading coefficient $\frac{a_1^2a_2^2a_3}{2!}$.  

Since $\gcd(d,a_3)=1$,  the congruence equation $a_3m\equiv r\pmod d$ has a unique solution modulo $d$. Consequently, there exists a unique integer   $m_i\in\{id+0,id+1,\cdots ,id+(d-1)\}$ such that $a_3m_i\equiv r\pmod d$ for each $0\leq i\leq\frac{a_1a_2}{d}-1$. Based on these observations, we can write

\[
\sum_{\substack{0\leq m\leq a_1a_2-1\\ d\mid r-ma_3}}N^p_{\{\frac{a_1}{d},\frac{a_2}{d}\}}\left(\frac{a_1}{d}\frac{a_2}{d}\frac{a_3d^2}{d}l+\frac{r-ma_3}{d}\right)\ \ \ \ \ \ \ \ \ \ \ \ \ \ \ \ \ \ \ \ \ \ \ \ \ \ \ \ 
\]

\begin{align*}
&=\sum_{i=0}^{\frac{a_1a_2}{d}-1}N^p_{\{\frac{a_1}{d},\frac{a_2}{d}\}}\left(\frac{a_1}{d}\frac{a_2}{d}\frac{a_3d^2}{d}l+\frac{r-m_ia_3}{d}\right)\\
&=\sum_{i=0}^{\frac{a_1a_2}{d}-1}N^p_{\{\frac{a_1}{d},\frac{a_2}{d}\}}\left(\frac{a_1}{d}\frac{a_2}{d}\left(a_3dl+q_i\right)+r_i\right),
\end{align*}
where $q_i$ and $r_i$ were uniquely determined (in view of division algorithm) satisfying the relation:$\frac{r-m_ia_3}{d}=\frac{a_1}{d}\frac{a_2}{d}q_i+r_i$. Since $\gcd(\frac{a_1}{d},\frac{a_2}{d})=1$, from the case $|A|=2$, we have that \[N^p_{\{\frac{a_1}{d},\frac{a_2}{d}\}}\left(\frac{a_1}{d}\frac{a_2}{d}\left(a_3dl+q_i\right)+r_i\right)\] is a polynomial in $a_3dl+q_i$ of degree 2 with the leading coefficient $\frac{\frac{a_1}{d}+\frac{a_2}{d}}{2!}$. This in turn gives that $N^p_{\{\frac{a_1}{d},\frac{a_2}{d}\}}\left(\frac{a_1}{d}\frac{a_2}{d}\left(a_3dl+q_i\right)+r_i\right)$ is a polynomial in $l$ of degree 2 with the leading coefficient $\frac{\frac{a_1}{d}+\frac{a_2}{d}}{2!}a_3^2d^2$. Consequently, \[\sum_{i=0}^{\frac{a_1a_2}{d}-1}N^p_{\{\frac{a_1}{d},\frac{a_2}{d}\}}\left(\frac{a_1}{d}\frac{a_2}{d}\left(a_3dl+q_i\right)+r_i\right)\] is a polynomial in $l$ of degree 2 with the leading coefficient 
$
\frac{a_1a_2}{d}\frac{\frac{a_1}{d}+\frac{a_2}{d}}{2!}a_3^2d^2
=\frac{a_1^2a_2a_3^2+a_1a_2^2a_3^2}{2!}$.

Now in view of Equation (\ref{eqn1}), we have that:
$
N_A^p(a_1a_2a_3l+r)-N_A^p(a_1a_2a_3(l-1)+r)
$ is a polynomial in $l$ of degree 2 with the leading coefficient $\frac{a_1^2a_2^2a_3+a_1^2a_2a_3^2+a_1a_2^2a_3^2}{2!}$.

From this we infer that $N_A^p(a_1a_2a_3l+r)$ is a polynomial in $l$ of degree 3. Now if one assumes $c_3$ to be the leading coefficient of $N_A^p(a_1a_2a_3l+r)$, then in accordance with the previous observations, we have
\[
3c_3=\frac{a_1^2a_2^2a_3+a_1^2a_2a_3^2+a_1a_2^2a_3^2}{2!}.
\]
This gives 
\[c_3=\frac{a_1^2a_2^2a_3+a_1^2a_2a_3^2+a_1a_2^2a_3^2}{3!}.
\]
Hence, the assertion is true for $k=3$. 

Assume that the assertion is true up to some $k-1\geq 3$. Now we shall prove the assertion for $k$. 

Let $A=\{a_1,a_2,\cdots ,a_k\}$ be a set of positive integers such that $\gcd(a_1,a_2,\cdots ,a_k)$\ $=1$. Let $d=\gcd(a_1,a_2,\cdots ,a_{k-1})$. Fix $r\in\{0,1,\cdots ,a_1a_2\cdots a_k-1\}$. Repeated application of Theorem \ref{roverA} $a_1a_2\cdots a_{k-1}$ times gives

$N_A^p(a_1a_2\cdots a_kl+r)-N_A^p(a_1a_2\cdots a_k(l-1)+r)$
\begin{align*}
&=\sum_{i=1}^{a_1a_2\cdots a_{k-1}}p_A(a_1a_2\cdots a_kl+r-ia_k)\\ &\ \ \ \ \ +\sum_{\substack{0\leq m\leq a_1a_2\cdots a_{k-1}-1\\ d|r-ma_k}}N_{A\setminus\{a_k\}}^p(a_1a_2\cdots a_kl+r-ma_k)\\
&=\sum_{i=1}^{a_1a_2\cdots a_{k-1}}p_A(a_1a_2\cdots a_kl+r-ia_k)\\ &\ \ \ \ \ +\sum_{\substack{0\leq m\leq a_1a_2\cdots a_{k-1}-1\\ d|r-ma_k}}N_{\{\frac{a_1}{d},\frac{a_2}{d},\cdots ,\frac{a_{k-1}}{d}\}}^p\left(\frac{a_1a_2\cdots a_kl}{d}+\frac{r-ma_k}{d}\right)\\
&=\sum_{i=1}^{a_1a_2\cdots a_{k-1}}p_A(a_1a_2\cdots a_kl+r-ia_k)
\end{align*}
\begin{align}\label{eqn2}
\ \ \ \ \ \ \ \ \ \  +\sum_{\substack{0\leq m\leq a_1a_2\cdots a_{k-1}-1\\ d|r-ma_k}}N_{\{\frac{a_1}{d},\frac{a_2}{d},\cdots ,\frac{a_{k-1}}{d}\}}^p\left(\frac{a_1}{d}\frac{a_2}{d}\cdots \frac{a_{k-1}}{d}\frac{d^{k-1}a_kl}{d}+\frac{r-ma_k}{d}\right).
\end{align}
By Lemma \ref{lemma2}, $p_A(a_1a_2\cdots a_kl+r-ia_k)$ is a polynomial in $l$ or $l-1$ (depending respectively on the bounds: $r-ia_k\geq 0$ or $r-ia_k<0$) of degree $k-1$ with the leading coefficient $\frac{(a_1a_2\cdots a_k)^{k-2}}{(k-1)!}$. In both the cases, $p_A(a_1a_2\cdots a_kl+r-ia_k)$ is  a polynomial in $l$ of degree $k-1$ with the leading coefficient $\frac{(a_1a_2\cdots a_k)^{k-2}}{(k-1)!}$. Consequently, $\sum_{i=1}^{a_1a_2\cdots a_{k-1}}p_A(a_1a_2\cdots a_kl+r-ia_k)$ is a polynomial in $l$ of degree $k-1$ with the leading coefficient $(a_1a_2\cdots a_{k-1})\frac{(a_1a_2\cdots a_k)^{k-2}}{(k-1)!}=\frac{(a_1a_2\cdots a_{k-1})^{k-1}a_k^{k-2}}{(k-1)!}$.

Since $\gcd(d,a_k)=1$, the equation $a_km\equiv r\pmod d$ has a unique solution modulo $d$. Consequently, there exists a unique integer $m_t\in\{td+0,td+1,\ldots,td+(d-1)\}$ such that $a_km_t\equiv r\pmod d$ for each $t\in\{0,1,\ldots , a_1a_2\cdots a_{k-1}-1\}$. In view of division algorithm, one can find a unique pair of integers $(q_t,r_t)$ such that $\frac{r-m_ta_k}{d}=\frac{a_1a_2\cdots a_{k-1}}{d^{k-1}}q_t+r_t$. Based on these observations, we can write 
\[\sum_{\substack{0\leq m\leq a_1a_2\cdots a_{k-1}-1\\ d|r-ma_k}}N_{\{\frac{a_1}{d},\frac{a_2}{d},\cdots ,\frac{a_{k-1}}{d}\}}^p\left(\frac{a_1}{d}\frac{a_2}{d}\cdots \frac{a_{k-1}}{d}\frac{d^{k-1}a_kl}{d}+\frac{r-ma_k}{d}\right)\ \ \ \ \ 
\]
\[=\sum_{t=0}^{\frac{a_1a_2\cdots a_{k-1}}{d}-1}N_{\{\frac{a_1}{d},\frac{a_2}{d},\cdots ,\frac{a_{k-1}}{d}\}}^p\left(\frac{a_1}{d}\frac{a_2}{d}\cdots \frac{a_{k-1}}{d}\left(d^{k-2}a_kl+q_t\right)+r_t\right).
\]
Since $\gcd\left(\frac{a_1}{d},\frac{a_2}{d},\cdots ,\frac{a_{k-1}}{d}\right)=1$, by induction assumption, we have that 
\[
N_{\{\frac{a_1}{d},\frac{a_2}{d},\cdots ,\frac{a_{k-1}}{d}\}}^p\left(\frac{a_1}{d}\frac{a_2}{d}\cdots \frac{a_{k-1}}{d}\left(d^{k-2}a_kl+q_t\right)+r_t\right)
\]
is a polynomial in $d^{k-2}a_kl+q_t$ of degree $k-1$ with the leading coefficient 
\[\frac{\sum_{s=1}^{k-1}\left(\frac{a_s}{d}\right)^{k-3}\prod_{\substack{j\neq s\\ 0\leq j\leq k-1}}\left(\frac{a_j}{d}\right)^{k-2}}{(k-1)!}.
\]
Consequently, $N_{\{\frac{a_1}{d},\frac{a_2}{d},\cdots ,\frac{a_{k-1}}{d}\}}^p\left(\frac{a_1}{d}\frac{a_2}{d}\cdots \frac{a_{k-1}}{d}\left(d^{k-2}a_kl+q_t\right)+r_t\right)$ is a polynomial in $l$ of degree $k-1$ with the leading coefficient \[\frac{(a_kd^{k-2})^{k-1}}{d^{k-3}d^{(k-2)(k-2)}}\frac{\sum_{s=1}^{k-1}a_s^{k-3}\prod_{\substack{j\neq s\\ 0\leq j\leq k-1}}a_j^{k-2}}{(k-1)!}.\]
This gives that \[\sum_{t=0}^{\frac{a_1a_2\cdots a_{k-1}}{d}-1}N_{\{\frac{a_1}{d},\frac{a_2}{d},\cdots ,\frac{a_{k-1}}{d}\}}^p\left(\frac{a_1}{d}\frac{a_2}{d}\cdots \frac{a_{k-1}}{d}\left(d^{k-2}a_kl+q_i\right)+r_i\right)\]
is a polynomial in $l$ of degree $k-1$ with the leading coefficient 
\begin{flushleft}
$\frac{a_1a_2\cdots a_{k-1}}{d}\frac{(a_kd^{k-2})^{k-1}}{d^{k-3}d^{(k-2)(k-2)}}\frac{\sum_{s=1}^{k-1}a_s^{k-3}\prod_{\substack{j\neq s\\ 0\leq j\leq k-1}}a_j^{k-2}}{(k-1)!}
=\frac{a_k^{k-1}\sum_{s=1}^{k-1}a_s^{k-2}\prod_{\substack{j\neq s\\ 0\leq j\leq k-1}}a_j^{k-1}}{(k-1)!}$
\end{flushleft}
\begin{flushright} 
$=\frac{\sum_{s=1}^{k-1}a_s^{k-2}\prod_{\substack{j\neq s\\ 0\leq j\leq k}}a_j^{k-1}}{(k-1)!}.\ \ \ \ \ \ \ \ \ \ \ \ $
\end{flushright}
Now in view of (\ref{eqn2}), we get $N_A^p(a_1a_2\cdots a_kl+r)-N_A^p(a_1a_2\cdots a_k(l-1)+r)$ as a polynomial in $l$ of degree $k-1$ with the leading coefficient
\begin{align*}
\frac{\sum_{s=1}^{k-1}a_s^{k-2}\prod_{\substack{j\neq s\\ 0\leq j\leq k}}a_j^{k-1}}{(k-1)!}+\frac{a_k^{k-2}(a_1a_2\cdots a_{k-1})^{k-1}}{(k-1)!}=\frac{\sum_{i=1}^{k}a_i^{k-2}\prod_{\substack{j\neq i\\ 0\leq j\leq k}}a_j^{k-1}}{(k-1)!}.
\end{align*} 
Let $c_k$ be the leading coefficient of $N_A^p(a_1a_2\cdots a_kl+r)$. Then from the above observations, we have
\[kc_k=\frac{\sum_{i=1}^{k}a_i^{k-2}\prod_{\substack{j\neq i\\ 0\leq j\leq k}}a_j^{k-1}}{(k-1)!},
\] 
which gives 
\[c_k=\frac{\sum_{i=1}^{k}a_i^{k-2}\prod_{\substack{j\neq i\\ 0\leq j\leq k}}a_j^{k-1}}{k!}.
\] 
Hence, the main assertion is true for $k$. Then the targeted estimate follows immediately while adhering to the discussion at the starting part of this proof. 
\end{proof}

\noindent {\bf Acknowledgement.} I would like to thank the editor and the referee for their constructive comments which helped improve the presentation and content of this paper. 
\bibliographystyle{amsplain}

\end{document}